\documentclass[11pt,a4paper]{article}
\usepackage{epsf,epsfig,amsfonts,amsgen,amsmath,amstext,amsbsy,amsopn,amsthm,lineno}
\usepackage{color}
\usepackage{graphicx}
\newtheorem{thm}{Theorem}

\newtheorem{lemma}{Lemma}

\theoremstyle{definition}

\newtheorem{claim}{Claim}

\newtheorem{remark}[claim]{Remark}

\renewcommand{\mod}{{\rm mod}}

\begin{document}
\title
{\bf\Large An inductive proof of Straub's $q$-analogue of Ljunggren's congruence\thanks{Supported by the Doctorate Foundation of Northwestern Polytechnical University (cx201326).}}

\date{}
\author{\small Bo Ning\thanks{E-mail address: ningbo\_math84@mail.nwpu.edu.cn (B. Ning)}\\[2mm]
\small Department of Applied Mathematics, Northwestern Polytechnical University,\\
\small Xi'an, Shaanxi 710072, P.R.~China}
\maketitle

\begin{abstract}
Recently, Straub gave an interesting $q$-analogue of a binomial congruence of Ljunggren. In this note we give an inductive proof of his result.
\medskip

\smallskip

\end{abstract}
\section{Introduction}

$q$-Series has been proved to be a challenging and interesting area in number theory. For a basic introduction to $q$-series and a wonderful survey paper, see \cite[Chapter 10]{Andrews_Askey_Roy} and \cite{Berndt}, respectively. In particular, $q$-analogues of a lot of classical congruences have been studies by several authors. We refer the readers to \cite{Andrews_5,Chapman_Pan,Clark,Pan_1,Pan_Cao,Shi_Pan,Straub}. For a detailed talk about $q$-congruences, we refer to Pan's Ph.D thesis \cite{Pan_0}.

As shown in \cite{Andrews_Askey_Roy}, we use $[n]_{q}:=1+q+q^2+\ldots +q^{n-1}=\frac{1-q^n}{1-q}$, ${[n]!}_q:=[n]_q[n-1]_q\cdots[1]_q$ and $\binom{n}{k}_q:=\frac{{[n]!}_q}{{[k]!}_q{[n-k]!}_q}$ to denote the usual $q$-analogues of numbers, factorials and binomial coefficients, respectively. It is easy to see that the usual numbers, factorials and binomial coefficients can be obtained as $q=1$.

The classical Lucas' congruence \cite{Lucas} tells us how to compute a binomial coefficient modulo a prime.

\begin{thm}[Lucas, \cite{Lucas}]\label{th1}
For any prime $p$, we can determine $\binom{n}{m}~(\mod~p)$ from the base $p$ expansions of $n$ and $m$. Specially, if $n=\sum_{i=0}^{t}b_ip^i$ and $m=\sum_{i=0}^{t}c_ip^{i}$ where $0\leq b_i,c_i<p$, then
\begin{align}\label{al1}
\binom{n}{m}\equiv\prod_{i=0}^{t}\binom{b_i}{c_i}~(\mod~p).
\end{align}
\end{thm}

In particular, when $n=kp$ and $m=sp$, (\ref{al1}) implies that $\binom{kp}{sp}\equiv \binom{k}{s} ~(\mod~p)$. For the case that a binomial coefficient modulo a prime power, Ljunggren \cite{Ljunggren} gave an interesting extension in 1952.

\begin{thm}[Ljunggren, \cite{Ljunggren}]\label{th2}
For any prime $p\geq 5$ and nonnegative integers $k$, $s$,
\begin{align}\label{al2}
\binom{kp}{sp}\equiv \binom{k}{s} ~(\mod~p^3).
\end{align}
\end{thm}

Recently, Straub \cite{Straub} gave a $q$-analogue of Ljunggren's binomial congruence (\ref{al2}).

\begin{thm}[Straub, \cite{Straub}]\label{th3}
For any prime $p\geq 5$ and nonnegative integers $k$, $s$,
\begin{align}\label{al3}
\binom{kp}{sp}_q\equiv\binom{k}{s}_{q^{p^2}}-\binom{k}{s+1}\binom{s+1}{2}\frac{p^2-1}{12}(q^p-1)^2~(\mod~[p]_q^3).
\end{align}
\end{thm}
Note that Straub's proof largely depends on the method in \cite{Clark}. In this note we give an inductive proof of Straub's result.

\begin{remark}
For $q$-binomial coefficients, there is a combinatorial interpretation in terms of areas under lattice paths due to P\'{o}lya, see \cite[Vol.4, p.444]{Polya}. In \cite[Chapter 1, Problem 6 ($d$)]{Stanley}, Stanley gave a combinatorial proof of Theorem \ref{th2}. Maybe it is interesting to find a combinatorial proof of Theorem \ref{th3}.
\end{remark}

\section{An inductive proof of Theorem \ref{th3}}
The following two results are well-known (see \cite[(3.3.10)]{Andrews_0} and \cite{Desarmenien,Olive}).
\begin{lemma}{\bf {(The $q$-Chu-Vandermonde-formula)}}
For nonnegative integers $m$, $n$ and $h$,
\begin{align*}
\sum\limits_{k=0}^{h}\binom{n}{k}_q\binom{m}{h-k}_q=\binom{m+n}{h}_q.
\end{align*}
\end{lemma}

\begin{lemma}{\bf {(The $q$-Lucas-Theorem)}} For any prime $p$ and nonnegative integers $a$, $b$, $r$ and $s$ such
that $0\leq b,s\leq p-1$,
\begin{align*}
\binom{ap+b}{rp+s}_q\equiv\binom{a}{r}\binom{b}{s}_q~~~(\mod~[p]_q).
\end{align*}
\end{lemma}
The next Lemma (\cite[Lemma 5]{Straub}) is a big step of Straub's proof. We first give a new proof of this lemma.

\begin{lemma}\label{le3}
For any prime $p\geq 5$,
\begin{align}\label{al4}
\binom{2p}{p}_q\equiv [2]_{q^{p^2}}-\frac{p^2-1}{12}(q^p-1)^2 ~(\mod~[p]_q^3).
\end{align}
\end{lemma}
\begin{proof}
By the $q$-Chu-Vandermonde-formula,
$$
\binom{2p}{p}_q=\sum_{i=0}^{p}\binom{p}{i}_q^2q^{i^2}=1+q^{p^2}+\sum_{i=1}^{p-1}\binom{p}{i}_q^2q^{i^2}=[2]_{q^{p^2}}+\sum_{i=1}^{p-1}\binom{p}{i}_q^2q^{i^2}.
$$
Thus we need only show that $\sum_{i=1}^{p-1}\binom{p}{i}_q^2q^{i^2}$ is congruence $(\mod~{[p]_q}^3)$ to $-\frac{p^2-1}{12}(q^p-1)^2$. Since
$$
\binom{p}{i}_q^2q^{i^2}=(\frac{{[p]!}_q}{{[i]!}_q{[p-i]!}_q})^2q^{i^2}=[p]_q^2(\frac{[p-1]!_q}{[i]!_q[p-i]!_q})^2q^{i^2},
$$
we need only show that $\sum_{i=1}^{p-1}(\frac{{[p-1]!}_q}{{[i]!}_q{[p-i]}_q})^2q^{i^2}$ is congruence $(\mod~{[p]_q})$ to $-\frac{p^2-1}{12}(1-q)^2$. Noting that $q^p\equiv 1 ~(\mod [p]_q)$, we have
$$\begin{array}{lll}
&(\frac{{[p-1]!}_q}{{[i]!}_q{[p-i]}_q})^2q^{i^2}\\
&=(\frac{(1-q^{p-1})(1-q^{p-2})\cdots (1-q^{p-i+1})}{(1-q)(1-q^2)\cdots (1-q^{i})})^2q^{i^2}(1-q)^2\\
&=(\frac{(q-q^{p})(q^2-q^p)\cdots (q^{i-1}-q^p)}{(1-q)(1-q^2)\cdots (1-q^{i})})^2q^{i}(1-q)^2\\
\end{array}
$$$$
\begin{array}{lll}
&\equiv(\frac{(q-1)(q^2-1)\cdots (q^{i-1}-1)}{(1-q)(1-q^2)\cdots (1-q^{i})})^2q^{i}(1-q)^2 ~(\mod~{[p]_q}) & \\
&=\frac{q^i(1-q)^2}{(1-q^i)^2},
\end{array}
$$
and it implies that $\sum_{i=1}^{p-1}(\frac{{[p-1]!}_q}{{[i]!}_q{[p-i]}_q})^2q^{i^2}$ is congruence $(\mod~{[p]_q})$ to $\sum_{i=1}^{p-1}\frac{q^i(1-q)^2}{(1-q^i)^2}$. Hence we are done if $\sum_{i=1}^{p-1}\frac{q^i}{(1-q^i)^2}$ is congruence $(\mod~{[p]_q})$ to $-\frac{p^2-1}{12}$. In fact, this is a deformation of Lemma 2 in \cite{Shi_Pan} due to Shi and Pan. The proof is complete.
\end{proof}

As a second step of an inductive proof of Theorem \ref{th3}, the following lemma is needed.

\begin{lemma}\label{le4}
For any prime $p\geq 5$,
\begin{align}\label{al5}
\binom{kp}{p}_q\equiv \binom{k}{1}_{q^{p^2}}-\binom{k}{2}\frac{p^2-1}{12}(q^p-1)^2 ~(\mod~[p]_q^3).
\end{align}
\end{lemma}

\begin{proof}
For a given integer $k$, if $k=1$, the proposition is trivially true. If $k=2$, it can be deduced from Lemma \ref{le3}. Now we assume that $k\geq 3$. By the $q$-Chu-Vandermonde formula,
$$\begin{array}{lll}
L&=\binom{kp}{p}_q \\
&=\sum_{i=0}^{p}\binom{(k-1)p}{p-i}_q\binom{p}{i}_qq^{i((k-2)p+i)}\\
&= \binom{(k-1)p}{p}_q +q^{(k-1)p^2}+\sum_{i=1}^{p-1}\binom{(k-1)p}{p-i}_q\binom{p}{i}_qq^{i((k-2)p+i)}\\
&=\binom{(k-1)p}{p}_q +q^{(k-1)p^2} +\sum_{i=1}^{p-1}\binom{p}{i}_qq^{i((k-2)p+i)}\sum_{j=0}^{p-i}\binom{(k-2)p}{p-i-j}_q\binom{p}{j}_qq^{j((k-3)p+i+j)}\\
&=\binom{(k-1)p}{p}_q +q^{(k-1)p^2}+\sum_{i=1}^{p-1}\binom{p}{i}_q\binom{(k-2)p}{p-i}_qq^{i((k-2)p+i)}+\sum_{i=1}^{p-1}\binom{p}{i}_q\binom{p}{p-i}_qq^{p^2(k-2)+i^2}\\
&+\sum_{i=1}^{p-1}\sum_{j=1}^{p-i-1}\binom{p}{i}_q\binom{(k-2)p}{p-i-j}_q\binom{p}{j}_qq^{i((k-2)p+i)+j((k-3)p+i+j)}.
\end{array}
$$
Now let $s(i,j)=i((k-2)p+i)+j((k-3)p+i+j)$ and let
$$\begin{array}{lll}
& L_1=\binom{(k-1)p}{p}_q +q^{(k-1)p^2},\\
& L_2=\sum_{i=1}^{p-1}\binom{p}{i}_q\binom{(k-2)p}{p-i}_qq^{i((k-2)p+i)},\\
& L_3=\sum_{i=1}^{p-1}\binom{p}{i}_q\binom{p}{p-i}_qq^{p^2(k-2)+i^2},\\
& L_4=\sum_{i=1}^{p-1}\sum_{j=1}^{p-i-1}\binom{p}{i}_q\binom{(k-2)p}{p-i-j}_q\binom{p}{j}_qq^{s(i,j)}.
\end{array}
$$
By the induction hypothesis,
$$\begin{array}{lll}
L_1&\equiv\binom{k-1}{1}_{q^{p^2}}+q^{(k-1)p^2}-\binom{k-1}{2}\frac{p^2-1}{12}(q^p-1)^2 ~(\mod~[p]_q^3)\\
&=\binom{k}{1}_{q^{p^2}}-\binom{k-1}{2}\frac{p^2-1}{12}(q^p-1)^2.\\
\end{array}
$$
On the other hand, by the $q$-Lucas-Theorem, for $1\leq i\leq p-1$, $\binom{p}{i}_q\equiv\binom{(k-2)p}{p-i}_q\equiv 0 ~(\mod~[p]_q)$, and we also have $q^{i((k-2)p+i)}\equiv q^{i((k-3)p+i)}~(\mod~[p]_q)$. By the induction hypothesis,
$$\begin{array}{lll}
L_2&\equiv\sum_{i=1}^{p-1}\binom{p}{i}_q\binom{(k-2)p}{p-i}_qq^{i((k-3)p+i)}~(\mod~[p]_q^3)\\
&=\binom{(k-1)p}{p}_q-\binom{(k-2)p}{p}_q-q^{(k-2)p^2}\\
&\equiv (\binom{k-1}{1}_{q^{p^2}}-\binom{k-2}{1}_{q^{p^2}}-q^{(k-2)p^2})-(\binom{k-1}{2}-\binom{k-2}{2})\frac{p^2-1}{12}(q^p-1)^2 ~(\mod~[p]_q^3)\\
&=-(k-2)\frac{p^2-1}{12}(q^p-1)^2.
\end{array}
$$
Similarly, we have
$$\begin{array}{lll}
L_3&=\sum_{i=1}^{p-1}\binom{p}{i}_q\binom{p}{p-i}_qq^{p^2(k-2)+i^2}\\
&\equiv\sum_{i=1}^{p-1}\binom{p}{i}_q\binom{p}{p-i}_qq^{i^2}~(\mod~[p]_q^3)\\
&=\binom{2p}{p}_q-1-q^{p^2}\\
&\equiv [2]_{q^{p^2}}-1-q^{p^2}-\frac{p^2-1}{12}(q^p-1)^2 ~(\mod~[p]_q^3)\\
&=-\frac{p^2-1}{12}(q^p-1)^2
\end{array}
$$
and
$$
L_4\equiv 0~(\mod~[p]_q^3).
$$
Thus, we have
$$\begin{array}{lll}
L&=L_1+L_2+L_3+L_4\\
&\equiv\binom{k}{1}_{q^{p^2}}-\binom{k-1}{2}\frac{p^2-1}{12}(q^p-1)^2-(k-2)\frac{p^2-1}{12}(q^p-1)^2-\frac{p^2-1}{12}(q^p-1)^2 ~(\mod~[p]_q^3)\\
&=\binom{k}{1}_{q^{p^2}}-\binom{k}{2}\frac{p^2-1}{12}(q^p-1)^2.
\end{array}
$$
The proof is complete.
\end{proof}

\begin{remark}
Motivated by Wilson's theorem which states that $(p-1)!\equiv -1 ~(\mod~p)$ if $p$ is a prime, Wolstenholme \cite{Wolstenholme} proved that for primes $p\geq 5$,
\begin{align}\label{al6}
\binom{2p-1}{p-1}\equiv 1~(\mod~p^3).
\end{align}
Later, Glaisher \cite{Glaisher} improved Wolstenholme's result (\ref{al6}) by proving that if $p$ is a prime $\geq 5$, then
\begin{align}\label{al7}
\binom{mp+p-1}{p-1}\equiv 1 ~(\mod~p^3).
\end{align}
Note that (\ref{al4}) and (\ref{al5}) can be considered as $q$-analogues of Wolstenholme's congruence (\ref{al6}) and Glaisher's congruence (\ref{al7}), respectively.
\end{remark}

\noindent{}
{\bf {Proof of Theorem 3.}}
We use induction on $s$ and $k$ to give a proof. For a given integer $k$, if $s=0$, it is trivially true. If $s=1$, it can deduced from Lemma \ref{le4}. If $k\leq s$, the result is also right. Now we assume that $k>s\geq 2$ and for a fixed $s$, we induct on $k$. By the $q$-Chu-Vandermonde formula,
$$\begin{array}{lll}
L&=\binom{kp}{sp}_q \\
   &=\sum_{i=0}^{p}\binom{(k-1)p}{sp-i}_q\binom{p}{i}_qq^{i((k-s-1)p+i)}\\
   &= \binom{(k-1)p}{sp}_q +\binom{(k-1)p}{(s-1)p}_qq^{(k-s)p^2}+\sum_{i=1}^{p-1}\binom{(k-1)p}{sp-i}_q\binom{p}{i}_qq^{i((k-s-1)p+i)}\\
   &=\binom{(k-1)p}{sp}_q +\binom{(k-1)p}{(s-1)p}_qq^{(k-s)p^2} +\sum_{i=1}^{p-1}\binom{p}{i}_qq^{i((k-s-1)p+i)}\sum_{j=0}^{p}\binom{(k-2)p}{sp-i-j}_q\binom{p}{j}_qq^{j((k-2-s)p+i+j)}\\
   &=\binom{(k-1)p}{sp}_q +\binom{(k-1)p}{(s-1)p}_qq^{(k-s)p^2}+\sum_{i=1}^{p-1}\binom{p}{i}_q\binom{(k-2)p}{sp-i}_qq^{i((k-s-1)p+i)}+\sum_{i=1}^{p-1}\binom{p}{i}_q\binom{(k-2)p}{(s-1)p-i}_q\cdot\\
   &q^{(p+i)((k-1-s)p+i)}+\sum_{i=1}^{p-1}\sum_{j=1}^{p-1}\binom{p}{i}_q\binom{(k-2)p}{sp-i-j}_q\binom{p}{j}_qq^{i((k-1-s)p+i)+j((k-2-s)p+i+j)}.
\end{array}
$$
Now let $s(i,j)=i((k-1-s)p+i)+j((k-2-s)p+i+j)$ and let
$$\begin{array}{lll}
& L_1=\binom{(k-1)p}{sp}_q+\binom{(k-1)p}{(s-1)p}_qq^{(k-s)p^2},\\
& L_2=\sum_{i=1}^{p-1}\binom{p}{i}_q\binom{(k-2)p}{sp-i}_qq^{i((k-s-1)p+i)},\\
& L_3=\sum_{i=1}^{p-1}\binom{p}{i}_q\binom{(k-2)p}{(s-1)p-i}_qq^{(p+i)((k-1-s)p+i)},\\
& L_4=\sum_{i=1}^{p-1}\sum_{j=1}^{p-1}\binom{p}{i}_q\binom{(k-2)p}{sp-i-j}_q\binom{p}{j}_qq^{s(i,j)}.
\end{array}
$$
By the induction hypothesis,
$$\begin{array}{lll}
L_1&\equiv\binom{k-1}{s}_{q^{p^2}}+\binom{k-1}{s-1}_{q^{p^2}}q^{(k-s)p^2}-\{\binom{k-1}{s+1}\binom{s+1}{2}+\binom{k-1}{s}\binom{s}{2}q^{(k-s)p^2}\}\frac{(p^2-1)(1-q)^2}{12}[p]_q^2 ~~(\mod~[p]_q^3)\\
&=\binom{k}{s}_{q^{p^2}}-\{\binom{k-1}{s+1}\binom{s+1}{2}+\binom{k-1}{s}\binom{s}{2}\}\frac{(p^2-1)(1-q)^2}{12}[p]_q^2.\\
\end{array}
$$
On the other hand, for $1\leq i\leq p-1$, $\binom{p}{i}_q\equiv\binom{(k-2)p}{sp-i}_q\equiv 0 ~(\mod~[p]_q)$ and $q^{i((k-s-1)p+i)}\equiv q^{i((k-s-2)p+i)}~(\mod~[p]_q)$.
By the induction hypothesis,
$$\begin{array}{ll}
L_2&\equiv\sum_{i=1}^{p-1}\binom{p}{i}_q\binom{(k-2)p}{sp-i}_qq^{i((k-s-2)p+i)}~(\mod~[p]_q^3)\\
&=\{\sum_{i=0}^{p}\binom{p}{i}_q\binom{(k-2)p}{sp-i}_qq^{i((k-s-2)p+i)}\}-\binom{(k-2)p}{sp}_q-\binom{(k-2)p}{(s-1)p}_qq^{p^2(k-s-1)}\\
&=\binom{(k-1)p}{sp}_q-\binom{(k-2)p}{sp}_q-\binom{(k-2)p}{(s-1)p}_qq^{p^2(k-s-1)}\\
&\equiv\{\binom{k-1}{s}_{q^{p^2}}-\binom{k-2}{s}_{q^{p^2}}-\binom{k-2}{s-1}_{q^{p^2}}q^{p^2(k-s-1)}\}-\{\binom{k-1}{s+1}\binom{s+1}{2}-\binom{k-2}{s+1}\binom{s+1}{2}-\binom{k-2}{s}\binom{s}{2}\}\cdot\\
&\frac{(p^2-1)(1-q)^2}{12}[p]_q^2 ~(\mod~[p]_q^3)\\
&=-\{\binom{k-1}{s+1}\binom{s+1}{2}-\binom{k-2}{s+1}\binom{s+1}{2}-\binom{k-2}{s}\binom{s}{2}\}\cdot\frac{(p^2-1)(1-q)^2}{12}[p]_q^2\\
&=-\{\binom{k-2}{s}s\}\cdot\frac{(p^2-1)(1-q)^2}{12}[p]_q^2.
\end{array}
$$
Similarly, we have
$$\begin{array}{lll}
L_3&\equiv\sum_{i=1}^{p-1}\binom{p}{i}_q\binom{(k-2)p}{(s-1)p-i}_qq^{i((k-1-s)p+i)}~(\mod~[p]_q^3)\\
&=\{\sum_{i=0}^{p}\binom{p}{i}_q\binom{(k-2)p}{(s-1)p-i}_qq^{i((k-1-s)p+i)}\}-\binom{(k-2)p}{(s-1)p}_q-\binom{(k-2)p}{(s-2)p}_qq^{p^2(k-s)}\\
&=\binom{(k-1)p}{(s-1)p}_q-\binom{(k-2)p}{(s-1)p}_q-\binom{(k-2)p}{(s-2)p}_qq^{p^2(k-s)}\\
&\equiv\{\binom{k-1}{s-1}_{q^{p^2}}-\binom{k-2}{s-1}_{q^{p^2}}-\binom{k-2}{s-2}_{q^{p^2}}q^{p^2(k-s)}\}-\{\binom{k-1}{s}\binom{s}{2}-\binom{k-2}{s}\binom{s}{2}-\binom{k-2}{s-1}\binom{s-1}{2}\}\cdot\\
&\frac{(p^2-1)(1-q)^2}{12}[p]_q^2 ~(\mod~[p]_q^3)\\
&=-\{\binom{k-1}{s}\binom{s}{2}-\binom{k-2}{s}\binom{s}{2}-\binom{k-2}{s-1}\binom{s-1}{2}\}\cdot\frac{(p^2-1)(1-q)^2}{12}[p]_q^2 ~(\mod~[p]_q^3)\\
&=-\{\binom{k-2}{s-1}(s-1)\}\cdot\frac{(p^2-1)(1-q)^2}{12}[p]_q^2
\end{array}
$$
and
$$\begin{array}{lll}
L_4&=\sum_{i=1}^{p-1}\sum_{j=1}^{p-1}\binom{p}{i}_q\binom{(k-2)p}{sp-i-j}_q\binom{p}{j}_qq^{i((k-1-s)p+i)+j((k-2-s)p+i+j)}\\
&\equiv\sum_{i+j=p,i\geq 1,j\geq1}\binom{p}{i}_q\binom{p}{j}_q\binom{(k-2)p}{(s-1)p}_{q}q^{i(p-j)}~(\mod~[p]_q^3)\\
&=\{\sum_{i+j=p}\binom{p}{i}_q\binom{p}{j}_q\binom{(k-2)p}{(s-1)p}_{q}q^{i(p-j)}\}-\binom{(k-2)p}{(s-1)p}_{q}(1+q^{p^2})\\
&=\binom{2p}{p}_{p}\cdot\binom{(k-2)p}{(s-1)p}_{q}-\binom{(k-2)p}{(s-1)p}_{q}(1+q^{p^2})\\
&\equiv\{[2]_{q^{p^2}}-1-q^{p^2}\}\cdot\binom{(k-2)p}{(s-1)p}_{q}-\binom{k-2}{s-1}_{q^{p^2}}\cdot\frac{(p^2-1)(1-q)^2}{12}[p]_q^2~(\mod~[p]_q^3)\\
&\equiv-\binom{k-2}{s-1}\cdot\frac{(p^2-1)(1-q)^2}{12}[p]_q^2~(\mod~[p]_q^3).
\end{array}
$$
Note that
$$
\binom{k}{s+1}\binom{s+1}{2}=\binom{k-1}{s+1}\binom{s+1}{2}+\binom{k-1}{s}\binom{s}{2}+\binom{k-2}{s}s+\binom{k-2}{s-1}(s-1)+\binom{k-2}{s-1}.
$$
Thus we have
$$\begin{array}{lll}
L&=L_1+L_2+L_3+L_4\\
&\equiv\binom{k}{s}_{q^{p^2}}-\binom{k}{s+1}\binom{s+1}{2}\cdot\frac{(p^2-1)(1-q)^2}{12}[p]_q^2 ~(\mod~ [p]_q^3).
\end{array}
$$
The proof is complete. {\hfill$\Box$}

\section{Another $q$-analogue of Ljunggren's congruence}
Glaisher's congruence (\ref{al7}) can be written as
\begin{align}\label{al8}
(mp+1)(mp+2)\ldots (mp+p-1)\equiv (p-1)! ~(\mod~p^3).
\end{align}
In 1999, Andrews \cite{Andrews_5} gave a $q$-analogue (\ref{al9}) of Glaisher's congruence (\ref{al8}):
If $p$ is an odd prime and $m\geq 1$, then
\begin{align}\label{al9}
\frac{{(q^{mp+1};q)_{p-1}}-q^{mp(p-1)/2}(q;q)_{p-1}}{(1-q^{(m+1)p})(1-q^{mp})}\equiv\frac{(p^2-1)p}{24}~(\mod~[p]_q).
\end{align}
Recently, with the help of Andrews' $q$-analogue (\ref{al9}), Pan \cite[Lemma 3.1]{Pan_2} got a general $q$-analogue of Ljunggren's congruence (\ref{al2}). The following $q$-analogue can be deduced from his result.

\begin{thm}\label{th4}
For any prime $p\geq 5$ and nonnegative integers $k$, $s$,
\begin{align}\label{al3}
\binom{kp}{sp}_q\equiv q^{(k-s)s\binom{p}{2}}\cdot(\binom{k}{s}_{q^{p}}+k\binom{k}{s+1}\binom{s+1}{2}\frac{p^2-1}{12}(q^p-1)^2)~(\mod~[p]_q^3).
\end{align}
\end{thm}

\end{document}